\documentclass[10pt,twoside]{article}
\usepackage[utf8]{inputenc}

\title{\uppercase{\textbf{\large{$\bb{A}^1$-connectivity of motivic spaces}}}}

\author{\uppercase{Tess Bouis} \hspace{0.2cm}and\hspace{0.1cm} \uppercase{Arnab Kundu}}
\date{}
\usepackage[left=2.5cm, right=2.75cm, top=1in, bottom=1in, headheight=2cm]{geometry}

\usepackage{fancyhdr}
\usepackage{stmaryrd}

\usepackage{mathdots}

\usepackage{mathtools}

\usepackage{amsmath, amscd, amsfonts, amssymb, amsthm,todonotes,eurosym,enumitem,relsize}

\usepackage[utf8]{inputenc}
\usepackage{dirtytalk}
\usepackage[T1]{fontenc}
\usepackage{mathrsfs}

\usepackage[colorlinks]{hyperref}
\usepackage[figure,table]{hypcap}
\definecolor{imperialred}{RGB}{237, 41, 57}
\definecolor{royalblue}{RGB}{64, 106, 212}
\definecolor{link}{RGB}{11,0,128}
\definecolor{gren}{RGB}{32,130,63}
\hypersetup{
	bookmarksnumbered,
	pdfstartview={FitH},
    citecolor=royalblue,
 	linkcolor=imperialred,
 	urlcolor=link,
 	linktocpage = true
	pdfpagemode={UseOutlines}
}

\usepackage{xcolor}
\usepackage{comment}

\usepackage[object=pgfhan]{pgfornament}

\usepackage{graphicx}
\usepackage{array}
\usepackage{csquotes}
\usepackage{extarrows}

\usepackage{needspace}

\input{xy}
\xyoption{all}
\usepackage{tikz-cd}
\usepackage{textcomp}
\usepackage{colonequals}

\DeclareFontFamily{T1}{cbgreek}{}
\DeclareFontShape{T1}{cbgreek}{m}{n}{<-6>  grmn0500 <6-7> grmn0600 <7-8> grmn0700 <8-9> grmn0800 <9-10> grmn0900 <10-12> grmn1000 <12-17> grmn1200 <17-> grmn1728}{}
\DeclareSymbolFont{quadratics}{T1}{cbgreek}{m}{n}
\DeclareMathSymbol{\qoppa}{\mathord}{quadratics}{19}
\DeclareMathSymbol{\Qoppa}{\mathord}{quadratics}{21}

\DeclareMathAlphabet\matheu{U}{eus}{m}{n}
\DeclareMathAlphabet{\mathmm}{U}{mmambb}{m}{n}
\newcommand{\bb}[1]{\mathbb{#1}}

\newcommand{\bun}{\begin{itemize}}
\newcommand{\bn}{\begin{enumerate}}
\newcommand{\bt}{\begin{thm}}
\newcommand{\bl}{\begin{lem}}
\newcommand{\bp}{\begin{prop}}
\newcommand{\bc}{\begin{cor}}
\newcommand{\bd}{\begin{teqn}}
\newcommand{\bud}{\begin{teqn*}}	
\newcommand{\bs}{\begin{proof}}
\newcommand{\br}{\begin{rem}}
\newcommand{\bdf}{\begin{defn}}
\newcommand{\bcj}{\begin{conj}}

\newcommand{\eun}{\end{itemize}}
\newcommand{\en}{\end{enumerate}}
\newcommand{\et}{\end{thm}}
\newcommand{\el}{\end{lem}}
\newcommand{\ep}{\end{prop}}
\newcommand{\ec}{\end{cor}}
\newcommand{\ed}{\end{teqn}}
\newcommand{\eud}{\end{teqn*}}
\newcommand{\es}{\end{proof}}
\newcommand{\er}{\end{rem}}
\newcommand{\edf}{\end{defn}}
\newcommand{\ecj}{\end{conj}}

\DeclareMathOperator{\spec}{Spec}



\usepackage{libertinus}
\usepackage{sectsty}
\usepackage{fancyhdr}

\newlength{\outermargin} 
\newlength{\mar} 
\newlength{\len}
\newlength{\temp}
\newtheorem{theorem}{Theorem}[section]
\newtheorem{lemma}[theorem]{Lemma}
\newtheorem{corollary}[theorem]{Corollary}
\newtheorem{proposition}[theorem]{Proposition}

\theoremstyle{definition}
\newtheorem{remark}[theorem]{Remark}

\newtheorem{example}[theorem]{Example}

\newtheorem{definition}[theorem]{Definition}

\newtheorem{conjecture}[theorem]{Conjecture}

\usepackage{indentfirst}

\DeclareMathOperator{\N}{\mathbb{N}}

\renewcommand{\epsilon}{\varepsilon}
\DeclareMathOperator{\mot}{mot}
\DeclareMathOperator{\A1}{\mathbb{A}^1}
\let\H\relax
\DeclareMathOperator{\H}{\mathrm{H}}
\let\P\relax
\DeclareMathOperator{\P}{\mathcal{P}}
\renewcommand{\ge}{\geqslant}
\renewcommand{\le}{\leqslant}

\DeclareFontFamily{U}{MnSymbolC}{}
\DeclareFontShape{U}{MnSymbolC}{m}{n}{
	<-5.5> MnSymbolC5
	<5.5-6.5> MnSymbolC6
	<6.5-7.5> MnSymbolC7
	<7.5-8.5> MnSymbolC8
	<8.5-9.5> MnSymbolC9
	<9.5-11.5> MnSymbolC10
	<11.5-> MnSymbolCb12
}{}

\usepackage[bbgreekl]{mathbbol}
\DeclareSymbolFontAlphabet{\mathbb}{AMSb}
\DeclareSymbolFontAlphabet{\mathbbl}{bbold}

\newcommand{\ccite}[2]{\cite[#2]{#1}}
\newcommand{\stacks}[1]{\cite[\href{https://stacks.math.columbia.edu/tag/#1}{Tag #1}]{stacks-project}}
\usepackage{soul}
\usepackage{enumitem}
\numberwithin{equation}{theorem}

\begin{document}

\maketitle

\begingroup
\renewcommand{\thefootnote}{}

\footnotetext{\textit{Mathematics Subject Classification 2020}: 14F42 (primary); 14F35, 55Q05 (secondary).}
\footnotetext{\textit{Keywords}: motivic homotopy theory, unstable $\mathbb{A}^1$-connectivity, algebraic $K$-theory, Gersten injectivity, slice filtration.}

\endgroup
	
	\pagestyle{fancy}
	\fancyhead[EC]{TESS BOUIS AND ARNAB KUNDU}
	\fancyhead[OC]{\uppercase{$\bb{A}^1$-connectivity of motivic spaces}}
	\fancyfoot[C]{\thepage}
	
\begin{abstract}
We prove a version of Morel's unstable $\mathbb{A}^1$-connectivity theorem over arbitrary base schemes. In the stable setting, this recovers (and simplifies the proof of) the known connectivity bounds due to Morel, Schmidt--Strunk, Deshmukh--Hogadi--Kulkarni--Yadav, and Druzhinin, and extends them to possibly non-noetherian schemes. \hbox{Using} the recent work of Bachmann--Elmanto--Morrow, this also implies that the slice filtration on homotopy $K$\nobreakdash-theory is convergent for qcqs schemes of finite valuative dimension.
\end{abstract}

\section{Introduction}

Motivic homotopy theory is an analogue of homotopy theory in algebraic geometry, where the affine line plays the role of the unit interval. It was introduced by Morel and Voevodsky \cite{morel_homotopy_1999} as a way to import ideas from topology to the study of algebraic varieties, and ultimately led to the proof of the Bloch--Kato conjecture relating Milnor $K$-theory and Galois cohomology \cite{voevodsky_motivic_2011}. 

The basic objects of study in motivic homotopy theory are {\it motivic spaces}, {\it i.e.}, $\mathbb{A}^1$-invariant Nisnevich sheaves of spaces on smooth schemes over a base. 
To pass from Nisnevich sheaves to motivic spaces, one typically uses the localisation functor $L_{\text{mot}}$, which is however somewhat inexplicit. In particular, because it enforces simultaneously $\mathbb{A}^1$-invariance (which adds information in the homological direction) and Nisnevich descent (which adds information in the cohomological direction), it may {\it a priori} spread in infinitely many degrees in both directions. 

This article is motivated by the following conjecture of Morel.

\begin{conjecture}[Stable $\mathbb{A}^1$-connectivity, \ccite{morel_stable_a1_connectivity}{Conjecture~2} \cite{ayoub_counter-example_A1_connectivity} \cite{schmidt_strunk}]
    Let $S$ be a noetherian scheme of Krull dimension $d$, and \hbox{$n \geqslant 0$} be an integer. Then for any $n$-connective space $E \in \mathcal{P}_{\text{Nis}}(\text{Sm}_S)$, the motivic spectrum $L_{\text{mot}} E \in \text{SH}(S)$ is $(n-d)$\nobreakdash-connective.
\end{conjecture}

When $S$ is the spectrum of a field, this conjecture was proved by Morel in \cite{morel_introduction_A1_homotopy,morel_stable_a1_connectivity}, and serves as a foundation for several results in motivic homotopy theory over a field \cite{morel_a1_topology_field}. 
Although Morel's original conjecture required $S$ to be regular and did not mention the dimension $d$ of $S$, Ayoub explained why the regularity hypothesis on $S$ should not be related to stable $\mathbb{A}^1$-connectivity, and exhibited counterexamples to this conjecture when $S$ is a scheme of positive dimension \cite{ayoub_counter-example_A1_connectivity}. A corrected version of Morel's conjecture, where the connectivity bound is now shifted by the dimension of $S$ to take into account Ayoub's counterexamples, was then formulated by Schmidt and Strunk \cite{schmidt_strunk}. This {\it shifted stable $\mathbb{A}^1$-connectivity conjecture} was proved when $S$ is a Dedekind scheme with infinite residue fields by Schmidt--Strunk \cite{schmidt_strunk}, when $S$ is the spectrum of a noetherian domain with infinite residue fields by Deshmukh--Hogadi--Kulkarni--Yadav \cite{deshmukh-hogadi-kulkarni-yadav_presentation_lemma_noetherian_domains}, and for general noetherian schemes $S$ by Druzhinin \cite{druzhinin_a1_connectivity}. 

Our main result is an unstable refinement of Druzhinin's shifted stable $\mathbb{A}^1$-connectivity theorem. We formulate this result using the notion of valuative dimension, which is better-behaved than the Krull dimension for non-noetherian schemes. Note that the valuative dimension of a scheme is always at least equal to its Krull dimension, and that both notions agree on noetherian schemes.

\begin{theorem}[Shifted unstable $\mathbb{A}^1$-connectivity; see Theorem~\ref{theoremunstableconnectivity}]\label{theoremintrounstable}
    Let $S$ be a qcqs scheme of valuative dimension~$d$, and $n \geqslant 0$ be an integer. Then for any $n$-connective space $E \in \mathcal{P}_{\emph{Nis}}(\emph{Sm}_S)$, the motivic space $L_{\emph{mot}} E \in \emph{H}(S)$ is $(n-d)$-connective as an object of $\mathcal{P}_{\emph{Nis}}(\emph{Sm}_S)$.
\end{theorem}

Theorem~\ref{theoremintrounstable} was previously known when $S$ is the spectrum of a field. For perfect fields, this was first proved by Morel \cite[Theorem~5.38]{morel_a1_topology_field} as a consequence of the main results of \cite{morel_a1_topology_field}. 
More direct proofs, which also hold for non-perfect fields, were later on provided by Asok \cite{asok_unstable_2009} and Ayoub \cite{ayoub_A1-algebraic_2024}. Here, we follow Ayoub's strategy, and adapt it to the case of an arbitrary base scheme $S$. The two main ingredients in doing so are the results of Clausen--Mathew \cite{clausen_hyperdescent_2021} and the presentation lemma over a general base proved in \cite{bouis_beilinson_2025}; see the outline below for more details.

\subsection*{Outline of the proof.}

\vspace{-.4cm}

\[
\scalebox{1.06}{$
\smash{\substack{E\\ n\text{-connective}}}
\;\Longrightarrow\;
\smash{\substack{E\\ \text{perverse }n\text{-connective}}}
\;\underset{\text{Coro.\ 2.6}}{\Longrightarrow}\;
\smash{\substack{L_{\mathrm{mot}}E\\ \text{perverse }n\text{-connective}}}
\;\Longrightarrow\;
\smash{\substack{L_{\mathrm{mot}}E\\ \text{generically }(n-d)\text{-connective}}}
\;\underset{\text{Gersten}}{\Longrightarrow}\;
\smash{\substack{L_{\mathrm{mot}}E\\ (n-d)\text{-connective}}}
$}
\]

In Section~\ref{sectionperverse}, we introduce the notion of perverse $n$-connectivity for spaces, and prove that it is preserved by the motivic localisation functor $L_{\text{mot}}$ from spaces to motivic spaces (Corollary~\ref{corollaryLmotpreserveperverseconnectivity}). In Section~\ref{sectionunstabletheorem}, we use this result to prove the shifted unstable $\mathbb{A}^1$-connectivity theorem (Theorem~\ref{theoremunstableconnectivity}), as well as consequences for stable $\mathbb{A}^1$\nobreakdash-connectivity (Corollary~\ref{corollaryshiftedstable}) and convergence of the slice filtration on homotopy $K$-theory (Corollary~\ref{corollaryconvergenceslice}).

\vspace{-\parindent}
\hspace{\parindent}

\subsection*{Notation and Conventions.}

\vspace{-\parindent}
\hspace{\parindent}

We use the word ``anima'' for spaces/$\infty$-groupoids, and we denote by $\text{Ani}$ the $\infty$-category of anima. We write $\P(\mathcal{C})$ for the $\infty$-category of presheaves of anima on $\mathcal{C}$. If $\tau$ is a Grothendieck topology on $\mathcal{C}$, we denote by \hbox{$\P_{\tau}(\mathcal{C}) \subseteq \P(\mathcal{C})$} the full subcategory of $\tau$-sheaves. Following \cite{morel_homotopy_1999} (see \cite[Appendix~C]{hoyois_quadratic_2014} for a treatment without noetherian hypotheses), given a scheme $S$, a {\it space} is an object of the category $\mathcal{P}_{\text{Nis}}(\text{Sm}_S)$, the category of {\it motivic spaces} $\text{H}(S) \subseteq \mathcal{P}_{\text{Nis}}(\text{Sm}_S)$ is the full subcategory of $\mathbb{A}^1$-invariant spaces, and the category of {\it motivic spectra} $\text{SH}(S)$ is obtained by formally inverting $\mathbb{P}^1_S$ (seen as a motivic space pointed at infinity) in the category $\text{H}(S)_\ast$ of pointed motivic spaces. All spaces $E \in \mathcal{P}_{\text{Nis}}(\text{Sm}_S)$ are implicitly extended from smooth to ind-smooth $S$-schemes by taking cofiltered limits, so that it is possible to evaluate $E$ on (henselian) local ind-smooth $S$-schemes. 

\section{Perverse connectivity is preserved by $\bb{A}^1$-localisation}\label{sectionperverse}

In this section, we prove that the motivic localisation functor $L_{\mot}$ preserves perverse $n$-connective objects (Corollary~\ref{corollaryLmotpreserveperverseconnectivity}) as a consequence of a variant of a result of Clausen--Mathew (Theorem~\ref{theoremClausenMathew}). Our definition of perverse connectivity (Definition~\ref{definitionperverseconnectivity}) is inspired by the perverse $t$-structures in stable motivic homotopy theory, as constructed in \cite[Section~2.2.4]{ayoub_sixI_2007} and \cite[Section~2.3]{bondarko_dimensional_2017}, and by Ayoub's notion of weak connectivity \cite{ayoub_P^1_connectivity,ayoub_A1-algebraic_2024}. We formulate this definition using the notion of valuative dimension, as introduced by Jaffard for commutative rings \cite[Chapter~IV]{jaffard_theorie_1960} and generalised to schemes in \cite[Section~2.3]{elmanto_cdh_2021}. We refer to these references for the definition and main properties of the valuative dimension. Note that the only property of the valuative dimension that we will use, and which fails for the Krull dimension on general non-noetherian schemes, is the fact that $\dim_v(\mathbb{A}^1_X)=\dim_v(X)+1$ for any scheme $X$.

\begin{definition}[Perverse connectivity]\label{definitionperverseconnectivity}
    Let $S$ be a scheme, and $n \geqslant 0$ be an integer. A presheaf $E \in \mathcal{P}(\text{Sm}_S)$ is {\it perverse $n$-connective} if any henselian local essentially smooth $S$-scheme $X$ of valuative dimension $d$ satisfies that the anima $E(X)$ is $(n-d)$-connective. Similarly, a space $E \in \mathcal{P}_{\text{Nis}}(\text{Sm}_S)$ (resp. a motivic space $E \in \text{H}(S)$) is {\it perverse $n$-connective} if its underlying presheaf $E \in \mathcal{P}(\text{Sm}_S)$ is perverse $n$-connective. Note that this notion also makes sense for pointed objects. 
\end{definition}

\begin{remark}\label{remarkperverseconnectivityNisnevichsheafification}
    Let $S$ be a scheme, and $n \geqslant 0$. Being perverse $n$-connective is a Nisnevich-local condition. In particular, a presheaf $E \in \mathcal{P}(\text{Sm}_S)$ is perverse $n$-connective if and only if the associated space $L_{\text{Nis}} E \in \mathcal{P}_{\text{Nis}}(\text{Sm}_S)$ is perverse $n$-connective.
\end{remark}

The following lemma will be used in Theorem~\ref{theoremClausenMathew}.

\begin{lemma}\label{lemmaCMwithvaldim}
    Let $R$ be a local ring, and $I \subseteq R$ be a finitely generated ideal contained in the maximal ideal of $R$. If $\spec(R)$ is of valuative dimension $d$, then $\spec(R) \setminus V(I)$ is of valuative dimension at most $d-1$.
\end{lemma}

\begin{proof}
    By \cite[Proposition~2.3.2\,(1)]{elmanto_cdh_2021}, the statement is Zariski-local on $\spec(R) \setminus V(I)$, so it suffices to prove that $\dim_v(R[\tfrac{1}{f}]) + 1 \leqslant \dim_v(R)$ for $f \in I$. Let $\mathfrak{p}$ be a prime ideal of $R[\tfrac{1}{f}]$, which restricts to a prime ideal of $R$, which we also denote $\mathfrak{p}$. Given a valuation ring $V$ of rank $r$ such that $R[\tfrac{1}{f}]/\mathfrak{p} \subseteq V \subseteq \text{Frac}(R[\tfrac{1}{f}]/\mathfrak{p})$, there exists a valuation ring extension $V'$ of $V$ of rank $r+1$ such that $R/\mathfrak{p} \subseteq V' \subseteq \text{Frac}(R/\mathfrak{p})$. Consequently, 
    $$\dim_v(R[\tfrac{1}{f}]) + 1 = \sup_{\mathfrak{p} \subseteq R[\tfrac{1}{f}]} \dim_v(R[\tfrac{1}{f}]/\mathfrak{p}) + 1 \leqslant \sup_{\mathfrak{p} \subseteq R[\tfrac{1}{f}]} \dim_v(R/\mathfrak{p}) \leqslant \sup_{\mathfrak{p} \subseteq R} \dim_v(R/\mathfrak{p}) = \dim_v(R).$$
\end{proof}

\begin{theorem}\label{theoremClausenMathew}
    Let $S$ be a qcqs scheme, and $n \geqslant 0$ be an integer. Then for every perverse $n$-connective space \hbox{$E \in \mathcal{P}_{\mathrm{Nis}}(\mathrm{Sm}_S)$} and every essentially smooth $S$-scheme $X$ of valuative dimension $d$, the anima $E(X)$ is $(n-d)$\nobreakdash-con\-nective.
\end{theorem}

\begin{proof}
    First note that, replacing valuative dimension by Krull dimension everywhere, this statement is a result of Clausen--Mathew (\cite[Theorem~3.30]{clausen_hyperdescent_2021}, where we use that $\dim(\overline{\{x\}}) \le d - \dim(X_x)$ for every point $x \in X$). We adapt the argument of \cite[Theorem~3.30]{clausen_hyperdescent_2021}, and prove the desired statement by induction on the valuative dimension of $X$.

    If the scheme $X$ is local, then the argument of \cite[Theorem~3.30]{clausen_hyperdescent_2021} proves the desired claim, by using Lemma~\ref{lemmaCMwithvaldim} to implement the induction step.

    In general, the underlying topological space of the qcqs scheme $X$ is a spectral space (\stacks{094L}), which is of Krull dimension at most $d$ by \cite[Proposition~2.3.2\,(8)]{elmanto_cdh_2021}. The proof of \cite[Theorem~3.14]{clausen_hyperdescent_2021} then implies, using the same argument as in Lemma~\ref{lemmaCMwithvaldim} for the induction step, that the anima $E(X)$ is $(n-d)$-connective if, for every point $x \in X$, the anima $E(X_x)$ is $(n-\dim_v(X_x))$-connective. The desired result is then a consequence of the previous paragraph.
\end{proof}

Following \cite[Section~2.3]{morel_homotopy_1999} (see also \cite[Definition~4.23]{elden_antieau_primer}), given a scheme $S$, we denote by
$$\text{Sing}^{\mathbb{A}^1} : \mathcal{P}(\text{Sm}_S) \longrightarrow \mathcal{P}(\text{Sm}_S)$$
the singular construction, which takes a presheaf of anima $E$ to the presheaf of anima
$$\text{Sing}^{\mathbb{A}^1}(E) := \vert E(- \times \Delta^\bullet)\vert.$$

\begin{proposition}\label{propositioninifinitecomposition}
    For any qcqs scheme $S$, the localisation $L_{\emph{mot}} : \mathcal{P}_{\emph{Nis}}(\emph{Sm}_S) \rightarrow \H(S)$ is equivalent to the countable iteration $(L_{\emph{Nis}}\, \emph{Sing}^{\A1})^{\circ \N}$.
\end{proposition}

\begin{proof}
    A classical reference, although stated only in the case where $S$ is the spectrum of a field, is \cite[Proposition~$5.20$]{morel_a1_topology_field}. The more general case where $S$ is an arbitrary qcqs scheme is \cite[Theorem~4.27]{elden_antieau_primer}.
\end{proof}

\begin{corollary}\label{corollaryLmotpreserveperverseconnectivity}
    Let $S$ be a qcqs scheme. Then for every integer $n \geqslant 0$, the localisation $L_{\emph{mot}} : \mathcal{P}_{\emph{Nis}}(\emph{Sm}_S) \rightarrow \H(S)$ preserves perverse $n$-connective objects.
\end{corollary}

\begin{proof}
    By Proposition~\ref{propositioninifinitecomposition}, it suffices to prove that both the Nisnevich sheafification $L_{\text{Nis}}$ and the singular construction $\text{Sing}^{\mathbb{A}^1}$ preserve perverse $n$-connective objects. For the Nisnevich sheafification, this is Remark~\ref{remarkperverseconnectivityNisnevichsheafification}. For the singular construction, let $E \in \mathcal{P}(\text{Sm}_S)$ be a presheaf of anima, and $X$ be a henselian local essentially smooth $S$\nobreakdash-scheme of valuative dimension $d$. 
    By \cite[Proposition~2.3.2\,(7)]{elmanto_cdh_2021}, for every integer $m \geqslant 0$, the scheme $X \times \Delta^m$ is of valuative dimension $d+m$. By Theorem~\ref{theoremClausenMathew}, this implies that, for every integer $m \geqslant 0$, the anima $E(X \times \Delta^m)$ is $(n-d-m)$-connective. By \cite[Lemma~2.4.6]{ayoub_A1-algebraic_2024} (see also \cite[Page~110, Footnote~41]{bachmann_A^1-invariant_2025} for a more direct argument in the stable setting), this in turn implies that the geometric realisation $(\text{Sing}^{\mathbb{A}^1} E)(X) := \vert E(X \times \Delta^\bullet)\vert$ is $(n-d)$-connective.
\end{proof}

\section{Shifted unstable $\mathbb{A}^1$-connectivity}\label{sectionunstabletheorem}

In this section, we use Corollary~\ref{corollaryLmotpreserveperverseconnectivity} and a new Gersten injectivity result (Proposition~\ref{propositiongersteninjectivity}) to prove our main result (Theorem~\ref{theoremunstableconnectivity}), and review some expected consequences (Corollaries~\ref{corollaryshiftedstable} and~\ref{corollaryconvergenceslice}).

\needspace{2\baselineskip} 

\begin{definition}[Connectivity]\label{definitionconnectivity}
    Let $S$ be a scheme, and $n$ be an integer. 
    \begin{enumerate}
        \item For every integer $i \geqslant 0$, the $i^{th}$ \emph{homotopy sheaf} $\pi_i E$ of a pointed space $E \in \mathcal{P}_{\text{Nis}}(\text{Sm}_S)_\ast$ is the Nisnevich sheaf of pointed sets associated with $U \mapsto \pi_i(E(U))$. 
        \item A pointed space $E \in \mathcal{P}_{\text{Nis}}(\text{Sm}_S)_\ast$ is \emph{$n$-connective} if its $i^{th}$ homotopy sheaf $\pi_i E$ vanishes for all integers $i < n$. By convention, any pointed space is $n$-connective if $n \le 0$. Similarly, a pointed motivic space $E \in \H(S)_\ast$ is \emph{$n$-connective} if its underlying pointed space $E \in \mathcal{P}_{\text{Nis}}(\text{Sm}_S)_\ast$ is $n$-connective.
        \item A space $E \in \mathcal{P}_{\text{Nis}}(\text{Sm}_S)$ is \emph{$n$-connective} if for every scheme $X \in \text{Sm}_S$ and every point $x \in E(X)$, the pointed space $(E,x) \in \mathcal{P}_{\text{Nis}}((\text{Sm}_S)_{/X})_\ast$ is $n$-connective. Similarly, a motivic space $E \in \text{H}(S)$ is \emph{$n$-connective} if its underlying space $E \in \mathcal{P}_{\text{Nis}}(\text{Sm}_S)$ is $n$-connective.
    \end{enumerate}
\end{definition}

\begin{remark}\label{remarkperverseconnectiveimpliesconnective}
    Let $S$ be a scheme, and $n \geqslant 0$ be an integer. By definition, if $E \in \mathcal{P}_{\text{Nis}}(\text{Sm}_S)$ is an $n$-connective space and $X$ is a henselian local essentially smooth $S$-scheme, then $E(X)$ is $n$-connective. In particular, any $n$-connective space is also perverse $n$-connective.
\end{remark}

The following notion was first introduced in \cite{colliot-thelene_bloch_1997}.\footnote{See also \cite{elmanto_morrow} for the original reference where this was called deflatability.}

\begin{definition}[Deflatability]\label{definitiondeflatablepresheaves}
    Given a scheme $S$, a functor $F \colon \text{Sch}^{\text{qcqs,op}}_S \rightarrow \text{Ani}$ is called \emph{deflatable} if for any qcqs $S$-scheme $X$, there exists a functorial equivalence between the maps
    $$\begin{tikzcd}
        F(\mathbb{P}^1_X) \arrow[shift right=1.25ex]{rr}[swap]{\pi_X^\ast \, \circ \, \infty_X^\ast}\arrow[rr, shift left=1.25ex, "j_X^{\ast}"]&& F(\mathbb{A}^1_X),
    \end{tikzcd}$$
    where the maps $\infty_X$, $j_X$, and $\pi_X$ are defined in the following commutative diagram of schemes.
    \begin{equation*}\label{diag:notations_gersten_injectivity}
    \begin{tikzcd}\mathbb{A}^1_{X}\arrow{dr}[swap]{\pi_{X}}\arrow[r,hook,"j_{X}"]&\mathbb{P}^1_{X}\arrow[d]&X\arrow[hook']{l}[swap]{\infty_{X}}\arrow[ld,-,double equal sign distance,double]\\ &X\end{tikzcd}
    \end{equation*}
    Note that this definition also makes sense for functors taking values in pointed anima.
\end{definition}

\begin{example}\label{exampledeflatable}
    If a functor $F \colon \text{Sch}^{\text{qcqs,op}}_S \rightarrow \text{Ani}$ is $\mathbb{A}^1$-invariant, then $F$ is deflatable in the sense of Definition~\ref{definitiondeflatablepresheaves} (see for instance \cite[Remark~3.9]{bouis_beilinson_2025}). Note that any family of functors satisfying the $\mathbb{P}^1$-bundle formula is also deflatable (\cite{colliot-thelene_bloch_1997}, see also \cite[Lemma~3.12]{bouis_beilinson_2025}).
\end{example}

The following lemma will be used in Propositions~\ref{propositiongersteninjectivity} and~\ref{propositionrelativedimensionzero}.

\begin{lemma}\label{lemmairreducible}
    Let $S$ be a scheme, $X$ be a henselian local essentially smooth $S$-scheme, and $s \in S$ be the image of the closed point of $X$. Then the fibre $X_s$ of $X$ over $s$ is irreducible.
\end{lemma}

\begin{proof}
    The proof is the same as in the first paragraph of the proof of \cite[Theorem~3.2]{bouis_beilinson_2025}, where we first reduce to $S$ being the spectrum of a henselian local ring instead of a henselian valuation ring.
\end{proof}

\begin{proposition}[Gersten injectivity over a base]\label{propositiongersteninjectivity}
    Let $S$ be a scheme, $X$ be a henselian local essentially smooth $S$-scheme, and $s \in S$ be the image of the closed point of $X$. Then for every finitary deflatable Nisnevich sheaf \hbox{$F : \emph{Sch}_S^{\emph{qcqs,op}} \rightarrow \emph{Ani}_\ast$} and every integer $i \geqslant 0$, we have that
    $$\mathrm{ker}\big(\pi_i (F(X)) \longrightarrow \pi_i (F(X_\eta))\big) = \{\ast\},$$
    where $\eta \in X$ is the generic point of the irreducible scheme $X_s$ (Lemma~\ref{lemmairreducible}), and $X_\eta$ is the local scheme of $X$ at $\eta$.\footnote{In practice, one often knows the vanishing of $\pi_i F$ at the henselisation $X_\eta^h$, which implies that $\pi_i F$ vanishes on $X_\eta$, hence also on $X$.} 
\end{proposition}

\begin{proof}
    The main geometric ingredient here is the Nisnevich-local presentation lemma \cite[Corollary~2.25]{bouis_beilinson_2025}. The proof that this implies the desired Gersten injectivity statement is classical and goes back to \cite{colliot-thelene_bloch_1997}; we reproduce the proof here for convenience, following \cite[Theorem~3.2]{bouis_beilinson_2025}. 

    Assume that the henselian local essentially smooth $S$-scheme $X$ is the henselian local scheme $X_x^h$ of a smooth $S$-scheme $X$ at some point $x \in X$. Let $[\sigma] \in \pi_i(F(X_x^h))$ be a class whose image in $\pi_i(F((X_x^h)_\eta))$ vanishes. We want to prove that $[\sigma]$ itself vanishes. The presheaf $\pi_i(F)$ commutes with filtered colimits of rings and the statement is Nisnevich local on $X$, so we can assume that the class $[\sigma]$ comes from a class $[\sigma] \in \pi_i(F(X))$ which vanishes in $\pi_i(F(X \setminus Z))$, for $Z \hookrightarrow X$ a closed immersion of $S$-schemes satisfying $\dim(Z_s) < \dim(X_s)$.

    It now suffices to produce a Nisnevich neighbourhood $x \in X' \rightarrow X$ such that the pullback of the class $[\sigma']$ vanishes in $\pi_i(F(X'))$. By the presentation lemma \cite[Corollary~2.25]{bouis_beilinson_2025}, there exist a smooth $S$-scheme~$T$, a Nisnevich neighbourhood $X' \rightarrow X$ of $x$, and a Nisnevich square
    $$\begin{tikzcd}
        X' \setminus Z' \arrow[r,hook] \ar[d] & X' \arrow[d,"\psi"] \\
        \mathbb{A}^1_T \setminus \psi(Z') \arrow[r,hook] & \mathbb{A}^1_T
    \end{tikzcd}$$
    such that the induced morphism $Z' \xrightarrow{\cong} \psi(Z') \rightarrow T$ is finite. By Nisnevich excision, the class $[\sigma'] \in \pi_i(F(X'))$ then lifts to a class $[\sigma'] \in \pi_i(F(\mathbb{A}^1_T))$ which vanishes in $\pi_i(F(\mathbb{A}^1_T \setminus \psi(Z')))$. Again by Nisnevich excision (where we use that $\psi(Z')$ is finite over $T$), this class $[\sigma'] \in \pi_i(F(\mathbb{A}^1_T))$ further lifts to a class $[\sigma'] \in \pi_i(F(\mathbb{P}^1_T))$ which vanishes in $\pi_i(F(\mathbb{P}^1_T \setminus \psi(Z')))$. By construction, the map $\infty^\ast_T : \pi_i(F(\mathbb{P}^1_T)) \rightarrow \pi_i(F(T))$ factors through the map $\pi_i(F(\mathbb{P}^1_T)) \rightarrow \pi_i(F(\mathbb{P}^1_T \setminus \psi(Z')))$, so the pullback class $\infty^\ast_T [\sigma'] \in \pi_i(F(T))$ vanishes. By deflatability (Definition~\ref{definitiondeflatablepresheaves}), this implies that $[\sigma'] = j^\ast_T [\sigma'] \in \pi_i(F(\mathbb{A}^1_T))$ also vanishes, which concludes the proof.
\end{proof}

\begin{lemma}\label{lemmaA1localisedvaluativedimension}
    Let $R$ be a local integral domain of valuative dimension $d$, and $\mathfrak{m}$ be the maximal ideal of $R$. Then the localisation $R':=R[X]_{\mathfrak{m} \cdot R[X]}$ is also of valuative dimension $d$.
\end{lemma}

\begin{proof}
    Let $F$ be the fraction field of $R$, $F':=F(X)$ be the fraction field of $R'$, and let $V'$ be a valuation ring satisfying that $R' \subseteq V' \subseteq F'$. We want to prove that the rank of $V'$ is at most equal to $d$. Let $V$ be the valuation ring defined as the intersection of $V'$ and $F$. By construction, we know that $R \subseteq V \subseteq F$. In particular, by the assumption that $R$ is of valuative dimension $d$, this implies that the valuation ring $V$ is of rank at most $d$. The condition that $R' = R[X]_{\mathfrak{m}\cdot R[X]} \subseteq V'$ then implies that the valuation on $V'$ is the Gauss valuation induced by~$V$: $\vert \sum_i a_i X^i \vert_{V'} = \min_i \vert a_i \vert_V$ (see for instance~\ccite{phd-thesis}{proof of Lemma~3.10}). The value groups of $V$ and $V'$ are then naturally isomorphic, and in particular the rank of $V'$ is equal to that of $V$, which is at most equal to $d$.
\end{proof}

\begin{proposition}\label{propositionrelativedimensionzero}
    Let $S$ be a scheme of valuative dimension $d$, $X$ be a henselian local essentially smooth $S$-scheme, and $s \in S$ be the image of the closed point of $X$. Then the henselian local scheme $X_\eta^h$, where $\eta \in X$ is the generic point of the irreducible scheme $X_s$ (Lemma~\ref{lemmairreducible}), is of valuative dimension at most $d$.
\end{proposition}

\begin{proof}
    The statement is Zariski-local around $s \in S$, so we can assume that the scheme $S$ is local with closed point~$s$ (\cite[Proposition~2.3.2\,(3)]{elmanto_cdh_2021}). By construction, there exists a smooth $S$-scheme $Y$ and a generic point $y \in Y_s$ such that $X_\eta^h$ is the henselisation of the local scheme of $Y$ at $y$. By \stacks{02G8} and \cite[Proposition~2.3.5]{elmanto_cdh_2021}, if $Y' \rightarrow Y$ is an unramified morphism of schemes, then $\dim_v(Y') \leqslant \dim_v(Y)$. In particular, it suffices to prove that the localisation of $Y$ at $y$ is of valuative dimension at most that of $S$. Zariski-locally around $y \in Y$, the map $Y \rightarrow S$ factors as a composite $Y \rightarrow \mathbb{A}^m_S \rightarrow S$ for some integer $m \geqslant 0$ and where the map $Y \rightarrow \mathbb{A}^m_S$ is étale. Using again \stacks{02G8} and \cite[Proposition~2.3.5]{elmanto_cdh_2021}, we can assume that $Y=\mathbb{A}^m_S$ and $y$ is the generic point of the special fibre of $Y$. 
    By induction on the integer $m\geqslant 1$, we can further assume that $m=1$. Finally, in this case, by definition of the valuative dimension, the desired claim is a consequence of Lemma~\ref{lemmaA1localisedvaluativedimension}. 
\end{proof}

\begin{theorem}[Shifted unstable $\A1$-connectivity theorem]\label{theoremunstableconnectivity}
    Let $S$ be a qcqs scheme of valuative dimension $d$, and \hbox{$n \geqslant 0$} be an integer. Then for any $n$-connective space $E \in \mathcal{P}_{\emph{Nis}}(\emph{Sm}_S)$, the motivic space \hbox{$L_{\emph{mot}}\,E \in \H(S)$} is \hbox{$(n-d)$}\nobreakdash-con\-nective. 
\end{theorem}

\begin{proof}
    By Remark~\ref{remarkperverseconnectiveimpliesconnective}, the space $E \in \mathcal{P}_{\text{Nis}}(\text{Sm}_S)$ is perverse $n$-connective in the sense of Definition~\ref{definitionperverseconnectivity}. By Corollary~\ref{corollaryLmotpreserveperverseconnectivity}, this implies that the motivic space $L_{\text{mot}} E \in \text{H}(S)$ is perverse $n$-connective. In particular, 
    any henselian local essentially smooth $S$-scheme $X$ of valuative dimension $e$ satisfies that the anima $(L_{\text{mot}} E)(X)$ is $(n-e)$\nobreakdash-connective. By Example~\ref{exampledeflatable}, the $\mathbb{A}^1$-invariant Nisnevich sheaf $L_{\text{mot}} E$ is deflatable. By Proposition~\ref{propositiongersteninjectivity} (applied to the deflatable Nisnevich sheaf $L_{\text{mot}} E$) and Proposition~\ref{propositionrelativedimensionzero}, this implies that any henselian local essentially smooth $S$-scheme $X$ satisfies that the anima $(L_{\text{mot}}E)(X)$ is $(n-d)$-connective, {\it i.e.}, that the motivic space $L_{\text{mot}} E \in \H(S)$ is $(n-d)$\nobreakdash-connective.
\end{proof}

\begin{remark}\label{remark1connectivity}
    As a consequence of Proposition~\ref{propositioninifinitecomposition}, one can actually prove a stronger version of Theorem~\ref{theoremunstableconnectivity} when $n=1$. Namely, for any $1$-connective space $E \in \mathcal{P}_{\text{Nis}}(\text{Sm}_S)$, the motivic space $L_{\text{mot}}\,E \in \H(S)$ is $1$-connective (\cite[Corollary~$4.30$]{elden_antieau_primer}).
\end{remark}

\begin{corollary}
    Let $S$ be a noetherian scheme of Krull dimension $d$, and $n$ be an integer. Then for any $n$-connective space $E \in \mathcal{P}_{\emph{Nis}}(\emph{Sm}_S)$, the motivic space $L_{\emph{mot}} E \in \H(S)$ is $(n-d)$-connective.
\end{corollary}

\begin{proof}
    This is a consequence of Theorem~\ref{theoremunstableconnectivity} and of the fact that the valuative and Krull dimensions agree on locally noetherian schemes (\cite[Proposition~2.3.2\,(9)]{elmanto_cdh_2021}). 
\end{proof}

The following result was previously known, using different techniques, when $S$ is locally noetherian \cite[Theorem~2.7]{druzhinin_a1_connectivity}.

\begin{corollary}[Shifted stable $\mathbb{A}^1$-connectivity theorem]\label{corollaryshiftedstable}
    Let $S$ be a qcqs scheme of valuative dimension $d$, and $n \geqslant 0$ be an integer. Then for any $n$-connective pointed space $E \in \mathcal{P}_{\emph{Nis}}(\emph{Sm}_S)_\ast$, the motivic spectrum $L_{\emph{mot}} E \in \emph{SH}(S)$ is $(n-d)$-connective. 
\end{corollary}

\begin{proof}
    This is a consequence of the definition of $n$-connectivity for a motivic spectrum (see for instance \cite[Section~2.1]{druzhinin_a1_connectivity}) and of the shifted unstable $\mathbb{A}^1$-connectivity theorem (Theorem~\ref{theoremunstableconnectivity}). Note that if one rather starts from a spectra-valued sheaf $E \in \mathcal{P}_{\text{Nis}}(\text{Sm}_S;\text{Sp})$, then the same result still holds: either as a consequence of Theorem~\ref{theoremunstableconnectivity} if $n \ge d$, or by using Proposition~\ref{propositiongersteninjectivity} to reproduce the classical argument of \cite{morel_introduction_A1_homotopy} (see also \cite{schmidt_strunk,deshmukh-hogadi-kulkarni-yadav_presentation_lemma_noetherian_domains}).
\end{proof}

\begin{remark}
    The analogue of Corollary~\ref{corollaryshiftedstable} for motivic $S^1$-spectra $\text{SH}^{S^1}(S)$, as also studied in \cite{schmidt_strunk,deshmukh-hogadi-kulkarni-yadav_presentation_lemma_noetherian_domains,druzhinin_a1_connectivity}, is similarly a consequence of Theorem~\ref{theoremunstableconnectivity}.
\end{remark}

\begin{remark}[Perverse homotopy $t$-structure]
    Let $S$ be a qcqs scheme of finite valuative dimension. The category $\text{SH}(S)$ can be equipped with a \emph{homotopy $t$-structure}, which is controlled in a nice way, when $S=\text{Spec}(k)$ is the spectrum of a field, by Morel's stable connectivity theorem (\cite[Section~5.2]{morel_introduction_A1_homotopy}, see also \cite[Section~2.1]{hoyois_quadratic_2014}). Over a more general base, the necessity for a \emph{shifted} stable $\mathbb{A}^1$-connectivity theorem (\cite{ayoub_counter-example_A1_connectivity}) prevents us from having a similar description of this homotopy $t$-structure in terms of a vanishing of homotopy sheaves. In this remark, we sketch how to define a perverse homotopy $t$-structure on $\text{SH}(S)$ using the ideas developed in this paper. Let $\prescript{p}{}{\text{SH}(S)}_{\ge 0} \subseteq \text{SH}(S)$ be the subcategory spanned by motivic spectra $E$ such that, for every henselian local essentially smooth $S$-scheme $X$, $E(X)$ is $(-\dim_v(S_s^h))$-connective where $s \in S$ is the image of the closed point $x \in X$. The category $\prescript{p}{}{\text{SH}(S)}_{\ge 0}$ is closed under homotopy colimits and extensions, hence it corresponds to the nonnegative part of a $t$-structure $(\prescript{p}{}{\text{SH}(S)}_{\ge 0},\prescript{p}{}{\text{SH}(S)}_{\le 0})$ on $\text{SH}(S)$ (\cite[Proposition~1.4.4.11]{lurie_higher_2017}), which we call the \emph{perverse homotopy $t$-structure} on $\text{SH}(S)$. By Gersten injectivity over a base (Proposition~\ref{propositiongersteninjectivity}), if $E \in \mathcal{P}_{\text{Nis}}(\text{Sm}_S)_\ast$ is $n$-connective (in the sense of Definition~\ref{definitionconnectivity}), then $L_{\text{mot}} E \in \prescript{p}{}{\text{SH}(S)}_{\ge n}$. This property is a generalisation of Morel's description of the homotopy $t$-structure on $\text{SH}(S)$ when $S=\text{Spec}(k)$ (\cite[Theorem~4.3.4]{morel_introduction_A1_homotopy}).
\end{remark}

We refer to \cite[Section~3.3]{bachmann_A^1-invariant_2025} for the notation relevant to the following result, which was first obtained by Bachmann--Elmanto--Morrow as a consequence of their main results on $\mathbb{A}^1$-invariant motivic cohomology (\cite[Corollary~9.5]{bachmann_A^1-invariant_2025}).

\begin{corollary}[Convergence of the slice filtration, \cite{bachmann_A^1-invariant_2025}]\label{corollaryconvergenceslice}
    Let $S$ be a qcqs scheme of valuative dimension $d$. Then the motivic spectrum $\mathrm{kgl}_S \in \emph{SH}(S)$ is very effective and, for every integer $n \geqslant 0$, the Nisnevich sheaf of spectra $\omega^\infty \emph{Fil}^n_{\emph{slice}} \emph{KGL}_S$ is $(n-d)$-connective. In particular, if $S$ is a qcqs scheme of finite valuative dimension, then the slice filtration $\emph{Fil}^\star_{\emph{slice}} \emph{KGL}_S$ is convergent.
\end{corollary}

\begin{proof}
    The proof is exactly as in \cite[Proposition~3.44 and Lemma~4.8]{bachmann_A^1-invariant_2025} (where the proof is stated for regular noetherian $S$ of dimension at most one), using the shifted stable $\mathbb{A}^1$-connectivity theorem (Corollary~\ref{corollaryshiftedstable}).
\end{proof}

\subsection*{Acknowledgements.}

\vspace{-\parindent}
\hspace{\parindent}

The authors would like to express their sincere gratitude to Elden Elmanto for his constant support throughout this project and for sharing many insights, as well as, some unpublished notes. We would also like to thank Frédéric Déglise, Marc Hoyois, and Brian Shin for related discussions, Neeraj Deshmukh for comments on a draft of this paper, and the referee for several
helpful comments and suggestions.

This work was done while the first-named author 
was partially supported by 
the SFB 1085 Higher Invariants, the Institute for Advanced Study, and the Simons Foundation, and while the second-named author received funding from the NSERC Discovery grant RGPIN-2025-07114, ``\textit{Motivic cohomology: theory and applications}'' and the PNRR grant CF 44/14.11.2022 ``\textit{Cohomological
		Hall algebras of smooth surfaces and applications}'', and hospitality from Chennai Mathematical Institute.

\bibliographystyle{alpha}
	
{\footnotesize
\bibliography{biblio.bib}
}

\medskip

{\footnotesize 
\textsc{Institute for Advanced Study, Princeton, New Jersey, 08540 USA} \par  
  \textit{Email address:} \texttt{tbouis@ias.edu} \par
  \textit{URL:} \texttt{https://tessbouis.com} \par
  
  \addvspace{\medskipamount}

  \textsc{Simion Stoilow Institute of Mathematics of the Romanian Academy (IMAR), Bucharest, Romania} \par  
  \textit{Email address:} \texttt{akundu.math@gmail.com} \par
  \textit{URL:} \texttt{https://arnabkundu.com} \par
}

\end{document}